\documentclass[12pt]{amsart}
\usepackage{colordvi,graphicx,ulem}
\usepackage{amssymb,amsmath,hyperref}
\usepackage{bm} 
\numberwithin{equation}{section}

\newtheorem{theorem}{Theorem}[section]
\newtheorem{proposition}[theorem]{Proposition}
\newtheorem{lemma}[theorem]{Lemma}
\newtheorem{corollary}[theorem]{Corollary}
\theoremstyle{remark}
\newtheorem{example}[theorem]{Example}
\newtheorem{remark}[theorem]{Remark}

\newcommand{\C}{{\mathbb C}}
\renewcommand{\P}{{\mathbb P}}
\newcommand{\calM}{{\mathcal M}}
\newcommand{\calN}{{\mathcal N}}

\DeclareMathOperator{\Span}{span}

\DeclareMathOperator{\GL}{GL}

\DeclareMathOperator{\Mat}{Mat}
\DeclareMathOperator{\rank}{rank}

\newcommand{\Fl}{{\mathbb F}\ell}
\newcommand{\adot}{a_\bullet}
\newcommand{\bdot}{b_\bullet}
\newcommand{\Edot}{E_\bullet}
\newcommand{\Epdot}{E'_\bullet}
\newcommand{\Fdot}{F_\bullet}
\newcommand{\Fpdot}{F'_\bullet}
\newcommand{\Rdot}{R_\bullet}
\DeclareMathOperator{\Gr}{Gr}
\newcommand{\balpha}{{\bm \alpha}}
\newcommand{\be}{{\bf e}}

\newcommand{\I}{\raisebox{-1pt}{\includegraphics{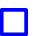}}}
\newcommand{\sI}{{\includegraphics{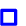}}}

\newcommand{\smMat}[2]{\bigl(\begin{smallmatrix}{#1}\\{#2}\end{smallmatrix}\bigr)}
\newcommand{\defcolor}[1]{\Blue{#1}}
\newcommand{\demph}[1]{\defcolor{{\sl #1}}}
\title[Certifiable Numerical Computation in Schubert Calculus]{A primal-dual formulation for certifiable computations in Schubert calculus} 
\author{Jonathan D. Hauenstein, Nickolas Hein, Frank Sottile}
\address{Jonathan D. Hauenstein\\
         Department of Applied and Computational Mathematics and Statistics\\
         University of Notre Dame\\
         Notre Dame\\
         Indiana \ 46556\\
         USA}
\email{hauenstein@nd.edu}
\urladdr{\url{http://www.nd.edu/~jhauenst}}
\address{Nickolas Hein \\
         Department of Mathematics\\
         University of Nebraska at Kearney\\
         Kearney\\
         Nebraska \ 68849\\
         USA}
\email{heinnj@unk.edu}
\urladdr{\url{http://www.unk.edu/academics/math/faculty/About_Nickolas_Hein/}}
\address{Frank Sottile \\
         Department of Mathematics\\
         Texas A\&M University\\
         College Station\\
         Texas \ 77843\\
         USA}
\email{sottile@math.tamu.edu}
\urladdr{\url{http://www.math.tamu.edu/~sottile/}}
\thanks{Research of Hauenstein supported in part by NSF grant DMS-1262428 and DARPA YFA.
        Research of Hein and Sottile supported in part by NSF grant DMS-0915211.}
\begin{document}
\begin{abstract}
 Formulating a Schubert problem as the solutions to a system of equations in
 either Pl\"ucker space or in the local coordinates of a Schubert cell
 typically involves more equations than variables.
 We present a novel primal-dual formulation of any Schubert problem on a Grassmannian or flag
 manifold as a system of bilinear equations with the same number of equations as variables.
 This formulation enables numerical computations in the Schubert calculus to be 
 certified using algorithms based on Smale's $\alpha$-theory.
\end{abstract}
\maketitle
%
\section*{Introduction}

Numerical algebraic geometry provides fast, efficient methods to approximate all solutions
to a system of polynomial equations~\cite{SW05}.
When the system is \demph{square} in that the number of equations is equal to the number
of variables, Smale's $\alpha$-theory~\cite[Ch.~8]{BCSS} gives methods to certify that
these approximate solutions correspond to actual solutions.
This is implemented in software which can be used to prove that all solutions have been
found and to certifiably count the number of real solutions~\cite{HS12}.

The Schubert calculus of enumerative geometry has come to mean all problems of determining
the linear subspaces of a vector space that have specified positions with respect to
other, fixed but general subspaces.
Originating in work of Schubert~\cite{Sch1879,Sch1886c,Sch1886a,Sch1886b}, it has long been an
active subject, undergoing significant recent development.
As a rich and well-understood class of geometric problems, the Schubert calculus is a
laboratory for the systematic study of new phenomena in enumerative
geometry~\cite{MSJ,IHP}, particularly as Schubert problems are readily formulated and
studied on a computer.

There are two traditional formulations of Schubert problems, one in global Pl\"ucker
coordinates and one in the local coordinates of a Schubert cell~\cite{Fulton}.
The first involves quadratic Pl\"ucker equations and linear equations, while the second
uses no Pl\"ucker equations and the linear equations are replaced by 
minors of matrices.
Both formulations typically involve far more equations than variables.
The second method was used to provide evidence for the Shapiro conjecture and discover its
generalizations~\cite{EG02,MT,MTV_Annals,MTV_R,So_FRSC} through the 
systematic study of many billions of instances of several thousand Schubert
problems~\cite{secant,monotone_MEGA,Lower,RSSS,So00}.
These were exact computations, using symbolic methods.
Numerical computation should enable the study
of far more and far larger Schubert problems as well as more delicate monodromy
computations on a computer.
For this hope to become a reality, algorithms tailored to Schubert calculus need to be
developed and implemented---this is starting to occur~\cite{HSS98,svv}---and methods to
certify computation need to be developed.

This second point was highlighted in a proof-of-concept paper to compute monodromy
(Galois groups) of some Schubert problems~\cite{LS}.
This included numerically solving a Schubert problem with 17589 solutions and
numerically computing enough monodromy permutations to conclude that the monodromy was the
full symmetric group $S_{17589}$.
These computations were not proofs in the ordinary sense as they relied upon software
heuristics for their conclusions.
This nevertheless inspired the development and implementation of methods to certify
computations in numerical algebraic geometry. 

Certification in numerical algebraic geometry rests upon work of Smale, who studied the
convergence of iterations of Newton's method applied to a point $x_0$, considered
to be an approximate solution to a system $F(x)=0$ of polynomial equations, when the
system is square~\cite{S86}.
This is called \demph{$\alpha$-theory} for the existence of an absolute constant $\alpha_0>0$ such
that when a constant $\alpha=\alpha(x_0,F)>0$ satisfies $\alpha<\alpha_0$, Newton iterates
starting at $x_0$ converge quadratically (doubling significant digits with each step) to a
solution to $F(x)=0$.
This was implemented in software~\cite{HS12} as an {\it a posteri} test to certify the
output of a numerical solver while 
Beltr{\'a}n  and Leykin~\cite{BL12,BL13} developed and implemented a certified
path-tracking algorithm which also certifies monodromy computations. 
In an important special case, {\it a posteri} estimates suffice to certify
monodromy~\cite{HHL}. 

Traditional formulations of Schubert problems 
typically lead to overdetermined systems (more equations than variables) so 
algorithms based on $\alpha$-theory cannot be used to
certify traditional numerical computation in Schubert calculus.
We present a novel formulation of any Schubert problem in primal-dual
coordinates as a square system of bilinear equations.
This will enable certification based on $\alpha$-theory, both {\it a posteri}
certification of approximate solutions and certified path-tracking.

The set of all $\ell$-dimensional linear subspaces ($\ell$-planes) in a vector space $V$ having
specified position with respect to a fixed subspace forms a \demph{Schubert subvariety} of
the Grassmannian.
A Schubert problem is formulated as the intersection of a collection of
Schubert varieties in a Grassmannian which are in general position.
A natural extension is to consider a Schubert problem to be
given by intersecting a collection of Schubert subvarieties in general position in a flag
manifold. 
We give a square primal-dual formulation of any Schubert problem on any 
flag manifold.

The main ideas are well illustrated for the Grassmannian.
In Section~\ref{S:Schubert_Calculus} we describe Schubert varieties in the Grassmannian and
give a traditional formulation of a Schubert problem as an overdetermined system of
determinantal equations in local coordinates for the Grassmannian.
We introduce our primal-dual reformulation in Section~\ref{S:dual}, where we use duality to
recast Schubert problems as a square system of bilinear equations in a larger space.
We improve this, using a hybrid approach and  more sophisticated local
coordinates to our primal-dual formulation in Section~\ref{S:improve}.
Finally, in Section~\ref{S:FlagVariety} we explain how this extends to all Schubert problems
in a flag manifold.

\section{Modeling Schubert problems}\label{S:Schubert_Calculus}

We describe Schubert varieties, Schubert problems, and how to model
them as systems of equations.
For a positive integer $n$, write $[n]$ for the set $\{1,\dotsc,n\}$.
Fix $\ell$ to be a positive integer at most $n{-}1$.
Let $w_0\in S_n$ be the permutation of $[n]$ such that $w_0(i)=n{+}1{-}i$.

\subsection{Schubert problems}

Let $V$ be an $n$-dimensional complex vector space.
The \demph{Grassmannian} \defcolor{$\Gr(\ell,V)$} of $\ell$-planes of $V$ is a complex
manifold of dimension $\ell(n{-}\ell)$ that is a subvariety of Pl\"ucker space,
$\P(\wedge^\ell V)\simeq \P^{\binom{n}{\ell}-1}$.
We often write $\Gr(\ell,n)$ to denote $\Gr(\ell,\C^n)$.

The Grassmannian has distinguished Schubert subvarieties.
A \demph{(complete) flag $\Fdot$} is a sequence 
\[
   \Fdot\ \colon\ 
    \{0\}\ \subsetneq\ F_1\ \subsetneq\ F_2\ \subsetneq
     \ \dotsb\ \subsetneq\ F_n\ =\ V\,,
\]
of linear subspaces of $V$ where $\dim F_i=i$.
The position of a $\ell$-plane $H$ with respect to the flag $\Fdot$ is 
the increasing sequence
$\alpha\colon 1\leq\alpha_1<\dotsb<\alpha_\ell\leq n$ where 
\[
   \alpha_i\ :=\ \min\{j\ \colon\ \dim(H\cap F_j) = i\}\,,\ \mbox{for\ }
      i=1,\dotsc,\ell\,.
\]
Write $\alpha:=\alpha(H,\Fdot)$ and call $\alpha$ a \demph{Schubert condition}. 
All such increasing sequences may occur.
Write \defcolor{$\binom{[n]}{\ell}$} for the set of Schubert conditions which is partially
ordered by coordinatewise comparison: $\alpha\leq\beta$ if and only if
$\alpha_i\leq\beta_i$ for each $i=1,\dotsc,\ell$.

Let $\defcolor{X^\circ_\alpha\Fdot}$ be the set of all $\ell$-planes $H$ with
$\alpha(H,\Fdot)=\alpha$, which is a \demph{Schubert cell}. 
This has dimension $\defcolor{\dim(\alpha)}:=(\alpha_1-1)+\dotsb+(\alpha_\ell-\ell)$. 
We have the Bruhat decomposition of the Grassmannian
\[
   \Gr(\ell,n)\ =\ \bigsqcup_{\alpha\in\binom{[n]}{k}} X^\circ_\alpha\Fdot\,.
\]

Given a Schubert condition $\alpha$ and a flag $\Fdot$, define the 
\demph{Schubert subvariety $X_\alpha\Fdot$} by 
 \begin{equation}\label{Eq:Schubert_Variety}
    X_\alpha\Fdot\ :=\ 
    \{H\in\Gr(\ell,V)\mid \dim(H\cap F_{\alpha_i})\geq i\,,\mbox{ for}\ i=1\dotsc,\ell\}\,,
 \end{equation}
which is the closure of the Schubert cell $X^\circ_\alpha\Fdot$.
This has dimension $\dim(\alpha)$ and thus codimension
$\defcolor{|\alpha|}:=\ell(n{-}\ell)-\dim(\alpha)=\ell(n{-}\ell)-\sum_i(\alpha_i-i)$.

\begin{example}
 Suppose that $\ell=2$, $n=8$, and $\Fdot$ is a complete flag in $\C^8$.
 Then
\[
   X_{(4,8)}\Fdot\ =\ 
    \{ H\in\Gr(2,8)\mid H\cap F_4\neq \{0\}\, \}\,,
\]
 as the remaining condition is that $H\subset F_8=\C^8$, which always holds.
 The codimension of $X_{(4,8)}\Fdot$ is 
 $2(8{-}2)-(4{-}1)-(8{-}2)=3$.\hfill{$\diamond$}
\end{example}

A \demph{Schubert problem} is a list $\balpha=(\alpha^1,\dotsc,\alpha^s)$ with
$\alpha^i\in\binom{[n]}{\ell}$ satisfying $|\alpha^1|+\dotsb+|\alpha^s|=\ell(n{-}\ell)$.
By Kleiman's Transversality Theorem~\cite{KL74}, if $\Fdot^1,\dotsc,\Fdot^s$ are general flags,
then the intersection
 \begin{equation}\label{Eq:SchubertProblem}
   X_{\alpha^1}\Fdot^1\ \cap\ 
   X_{\alpha^2}\Fdot^2\ \cap\ \dotsb\ \cap\ 
   X_{\alpha^s}\Fdot^s\,,
 \end{equation}
is transverse and consists of finitely many points.
Such an intersection~\eqref{Eq:SchubertProblem} is an \demph{instance} of the Schubert
problem $\balpha$.

\begin{example}\label{Ex:SP}
 Suppose that $\balpha=((4,8),(4,8),(4,8),(4,8))$, which is a Schubert problem on
 $\Gr(2,8)$. 
 Let $\Fdot^1,\Fdot^2,\Fdot^3,\Fdot^4$ be general flags.
 The corresponding instance is
 \begin{multline*}
   \qquad X_{(4,8)}\Fdot^1\,\cap\, X_{(4,8)}\Fdot^2\,\cap\,
   X_{(4,8)}\Fdot^3\,\cap\, X_{(4,8)}\Fdot^4
    \\ 
     =\ \{ H\in\Gr(2,8)\mid H\cap F^i_4\neq \{0\}\ i=1,2,3,4 \}\,. \qquad
 \end{multline*}
 That is, the 2-planes that have non-trivial intersection with four 4-planes in general
 position in $\C^8$.
 This can be shown to have four solutions.\hfill{$\diamond$}
\end{example}

We may use the global Pl\"ucker coordinates to formulate a Schubert problem as the
solutions to a system of equations.
In the projective space $\P^{\binom{n}{\ell}-1}$, the Grassmannian is defined by quadratic
Pl\"ucker equations and Schubert varieties are given by certain linear equations.
This is not considered to be an efficient encoding of the Schubert problem.

For instance, in the Schubert problem of Example~\ref{Ex:SP}, Pl\"ucker 
space has dimension $\binom{8}{2}-1=27$ and there are $\binom{8}{4}=70$ linearly
independent quadratic Pl\"ucker equations.
These cut out the twelve-dimensional Grassmannian $\Gr(2,8)$ and each Schubert variety
$X_{(4,8)}\Fdot$ is cut out by six independent linear 
equations, for a total of 94 equations.

\subsection{Local coordinates for Schubert varieties}

Local coordinates for Schubert varieties in $\Gr(\ell,V)$ are described
in~\cite[Ch.~10]{Fulton}. 
Let $M$ be a full rank $m\times n$ matrix with $m\leq n$.
If $\ell\leq m$, let \defcolor{$M_\ell$} be the first $\ell$ rows of $M$
and set $\defcolor{R_\ell(M)}$ to be the row span of $M_\ell$.
For this, fix a basis $\be_1,\dotsc,\be_n$ of $V\simeq\C^n$ corresponding to the
columns of $M$.
If $\adot=a_1<\dotsb<a_t\subset[n]$ and $a_t\leq m$ then \defcolor{$R_{\adot}(M)$} is the
(partial) flag of linear subspaces
\[
    R_{a_1}(M)\ \subsetneq\ 
    R_{a_2}(M)\ \subsetneq\ \dotsb\ \subsetneq\ 
    R_{a_t}(M)\,,
\]
and we write \defcolor{$\Rdot(M)$} for the complete flag obtained from an invertible
$n\times n$ matrix $M$.

If $N$ is an $n\times m$ matrix, we similarly have the column span \defcolor{$C_\ell(N)$}
and the flag $C_{\adot}(N)$.
These are subspaces of the dual vector space $V^*$ equipped with the basis
$\be^*_1,\dotsc,\be^*_n$ of $V^*$ dual to $\be_1,\dotsc,\be_n$, which corresponds to the
rows of $N$.

The association $M\mapsto R_{\ell}(M)$ realizes the open subset of rank $\ell$ matrices in
$\Mat_{\ell\times n}(\C)$ as a $GL(\ell,\C)$-principal bundle over the
Grassmannian, called the Stiefel manifold.
We identify subsets of the Stiefel manifold that parametrize Schubert
varieties and their intersections, and call these 
subsets \demph{Stiefel coordinates}.

Every $\ell$-plane $H$ is the row space of a unique unique echelon matrix, 
$M\in\Mat_{\ell\times n}(\C)$. 
That is, the last entry (the pivot) in row $i$ of $M$ is $1$, and if $\alpha_i$ is the
column of the pivot, then $\alpha_1<\dotsb<\alpha_\ell$ and the $1$ in position
$(i,\alpha_i)$ is the only nonzero element of its column.
Here are echelon matrices for $\alpha=(4,8)$ and $\alpha=(2,4,5,7)$ with
$(\ell,n)=(2,8)$ and $(4,7)$, respectively, and where $m_{i,j}$ represents some complex
number.  
 \begin{equation}\label{Eq:echelon_matrices}
   \left(\!\begin{array}{cccccccc} \!m_{1,1}&m_{1,2}&m_{1,3}&\!1&\!0&0&0&\!0\!\\
             \!m_{2,1}&m_{2,2}&m_{2,3}&\!0\!&m_{2,5}&m_{2,6}&m_{2,7}&\!1\!\end{array}\!\right)
    \!\qquad
   \left(\begin{array}{ccccccc} 
         \!m_{1,1}&\!1\!&  0    &\!0\!&\!0\!&0&\!0\!\\
        \!m_{2,1}&\!0\!&m_{2,3}&\!1\!&\!0\!&0&\!0\!\\
         \!m_{3,1}&\!0\!&m_{3,3}&\!0\!&\!1\!&0&\!0\!\\
         \!m_{4,1}&\!0\!&m_{4,3}&\!0\!&\!0\!&m_{4,6}&\!1\!
    \end{array}\right)
 \end{equation}

Let \defcolor{$\calM_\alpha$} be the set of echelon matrices with pivots $\alpha$.
The number of  unspecified entries in matrices in $\calM_\alpha$ is $\dim(\alpha)$, which shows
that $\calM_\alpha\simeq\C^{\dim(\alpha)}$.
Let $\be_1,\dotsc,\be_n$ be an ordered basis of $\C^n$ corresponding to the columns of our
matrices.
If we define the standard coordinate flag \defcolor{$\Edot$} so that 
$E_j:=\Span\{\be_1,\dotsc,\be_j\}$, then a $\ell$-plane $H$ has 
position $\alpha$ with respect to the flag $\Edot$, so that $H\in X^\circ_\alpha\Edot$,
if and only if its echelon matrix has pivots $\alpha$.
As distinct echelon matrices give distinct $\ell$-planes and vice-versa, 
this shows that $R_\ell\colon \calM_\alpha\to X^\circ_\alpha\Edot$ is an isomorphism and thus 
$X^\circ_\alpha\Edot\simeq\C^{\dim(\alpha)}$.
Since the Schubert cell is dense in the Schubert variety, $\calM_\alpha$
gives local coordinates for the Schubert variety~$X_\alpha\Edot$.

These echelon matrices $\calM_\alpha$ also give local coordinates for arbitrary Schubert
varieties. 
Let $\Phi$ be a full rank $n\times n$ matrix.
If $M\in\calM_\alpha$, then we invite the reader to check that 
$R_\ell(M\Phi)\in X^\circ_\alpha\Rdot(\Phi)$.

Let $\Epdot$ be the coordinate flag opposite to $\Edot$, where 
$E'_i:=\Span\{\be_{n+1-i},\dotsc,\be_n\}$ and 
$\dim E_i\cap E'_j=\max\{0,i{+}j{-}n\}$, so that $\Edot$ and $\Epdot$ are in linear
general position.
For $\beta\in\binom{[n]}{\ell}$, let $\calM^\beta$ be the set of echelon matrices in which row  
$j$ has its first non-zero entry of $1$ in column $n{+}1{-}\beta_{\ell{+}1{-}j}$, and this
entry is the only nonzero entry in its column, but the remaining entries are
unconstrained. 
Rotating a matrix $180^\circ$ gives a bijection between $\calM^\beta$ and $\calM_\beta$, and the
map $M\mapsto R_\ell(M)$ is a bijection between $\calM^\beta$ and the Schubert 
cell~$X^\circ_\beta\Epdot$. 
Here is a typical matrix in $\calM^{(3,7,9)}$ for $\Gr(3,9)$,
\[
   \left(\begin{array}{ccccccccc} 
         1&m_{1,2}&0&m_{1,4}&m_{1,5}&m_{1,6}&0&m_{1,8}&m_{1,9}\\
         0&   0  &1&m_{2,4}&m_{2,5}&m_{2,6}&0&m_{2,8}&m_{2,9}\\
         0&   0  &0&  0   &  0  &   0   &1&m_{3,8}&m_{3,9}
    \end{array}\right)\ .
\] 

Suppose that we have indices $\alpha,\beta\in\binom{[n]}{\ell}$ with $\alpha_i+\beta_{\ell+1-i}\leq
n$ for $i=1,\dotsc,a$.
Consider the set \defcolor{$\calM^\beta_\alpha$} of matrices $(m_{i,j})$ where 
 \begin{eqnarray*}
   m_{i,j}=0 &\mbox{if}& i<\beta_{\ell+1-i}\quad\mbox{or}\quad \alpha_i<j\\
   m_{i,j}=1 &\mbox{if}& i=\beta_{\ell+1-i}\ ,
 \end{eqnarray*}
and $m_{i,j}$ is otherwise unconstrained.
Here is a typical matrix in $M^{(5,7,8,11)}_{(6,8,10,11)}$,
\[
   \left(\begin{array}{ccccccccccc} 
         1&m_{1,2}&m_{1,3}&m_{1,4}&m_{1,5}&m_{1,6}&0&0&0&0&0\\
         0&   0  &0&1&m_{2,5}&m_{2,6}&m_{2,7}&m_{2,8}&0&0&0\\
         0&   0  &0&0&1&m_{3,6}&m_{3,7}&m_{3,8}&m_{3,9}&m_{3,10}&0\\
         0&   0  &0&0& 0 &0&  1  &m_{4,8}&m_{4,9}&m_{4,10}&m_{4,11}
    \end{array}\right)\ .
\]
If $M\in \calM^\beta_\alpha$, then $R_\ell(M)\in X_\alpha\Edot\cap X_\beta\Epdot$.
Indeed, as the last nonzero element $m_{i,\alpha_i}$ in row $i$ is in column $\alpha_i$, we
conclude that $H\in X_\alpha\Edot$. 
To see that $H\in X_\alpha\Epdot$, simply read the rows in reverse order.

A full rank $n\times n$ matrix $\Phi$ gives a complete flag $\Fdot=\Rdot(\Phi)$.
If we let $F'_i$ be the row span of the last $i$ rows of $\Phi$, we get a second flag
$\Fpdot=\Rdot(w_0\Phi)$ which is opposite to $\Fdot$ in that
$\dim F_i\cap F'_j=\max\{0,i{+}j{-}n\}$, and given any two such flags, there is a full rank
matrix $\Phi$ which gives rise to them.
We will refer to $\Fdot$ and $\Fpdot$ as the pair of opposite flags given by $\Phi$.
Note that $\Edot$ and $\Epdot$ are given by the identity matrix.

\begin{proposition}\label{P:localCoordinates}
 Let $\Fdot$ and $\Fpdot$ be opposite flags given by a full rank $n\times n$ matrix
 $\Phi$, $\Fdot=\Rdot(\Phi)$ and $\Fpdot=\Rdot(w_0\Phi)$. 
 The set $\calM_\alpha$ of echelon matrices with pivots $\alpha$ is an affine
 space isomorphic to the Schubert cell $X^\circ_\alpha\Fdot$ via the map
 $\calM_\alpha\ni M\mapsto R_\ell(M\Phi)$.

 When $\alpha_i+\beta_{\ell+1-i}\leq n$ for $i=1,\dotsc,\ell$, the set $\calM_\alpha^\beta$ is
 an affine space which parametrizes a dense subset of the intersection 
 $X_\alpha\Fdot\cap X_\beta\Fpdot$ via  the map $\calM_\alpha^\beta\ni M\mapsto R_\ell(M\Phi)$.
 (If $\alpha_i+\beta_{\ell+1-i}>n$ for some $i$, then both
 $\calM_\alpha^\beta$ and  $X_\alpha\Fdot\cap X_\beta\Fpdot$ are empty.)
\end{proposition}

It only remains to show that $\calM_\alpha^\beta$ parametrizes
a dense subset of $X_\alpha\Fdot\cap X_\beta\Fpdot$.
This follows as the map is injective on the locus where the entries $m_{i,\alpha_i}$ are
simultaneously nonzero and the two varieties are irreducible of the same dimension
$\ell(n{-}\ell){-}|\alpha|{-}|\beta|$. 
For the intersection of Schubert varieties, this is because the flags are in general
position, and for $\calM_\alpha^\beta$ this is verified by counting parameters.

\subsection{Determinantal equations for Schubert varieties}\label{SS:det}
Let $\alpha\in\binom{[n]}{\ell}$ and let $\Fdot=\Rdot(\Phi)$ be a flag given by a full
rank matrix $\Phi$.
The geometric conditions~\eqref{Eq:Schubert_Variety} defining a Schubert variety
$X_\alpha\Fdot$ give determinantal equations in local coordinates.

If $H= R_\ell(M)$, where $M\in\Mat_{\ell\times n}(\C)$, then 
$\dim(H\cap F_b)\geq c$ if and only if 
\[
   \rank\left( \begin{array}{c}M\\\Phi_b\end{array}\right)\ \leq\ \ell{+}b{-}c\,,
\]
where $\Phi_b$ denotes the submatrix of $\Phi$ given by the first $b$ rows.  
This is characterized by the vanishing of each of the square $\ell{+}b{-}c{+}1$
subdeterminants (minors) of $\smMat{M}{\Phi_b}$. 

\begin{proposition}\label{P:minor_eqns}
 Let $\calM$ be a set of matrices parametrizing a subset $Y$ of\/ $\Gr(\ell,\C^n)$, let $\Phi$ be
 a full-rank matrix, and let $\alpha\in\binom{[n]}{\ell}$.
 Then, in the coordinates $M\in\calM$ for $Y$, the intersection $Y\cap X_{\alpha}\Rdot(\Phi)$ of $Y$
 with the Schubert variety is defined by the vanishing of the square $\ell{+}\alpha_i{-}i{+}1$
 minors of\/ $\smMat{M}{\Phi_{\alpha_i}}$, for $i=1,\dotsc,\ell$.
\end{proposition}

This set \defcolor{$I(\alpha,\Phi)$} of minors defining $X_\alpha\Fdot$ is typically more
than is necessary.
For $\alpha\in\binom{[n]}{\ell}$, let \defcolor{$\|\alpha\|$} be the number of sequences 
$\beta\in\binom{[n]}{\ell}$ such that $\beta\not\leq\alpha$.

\begin{lemma}\label{L:Number_eqs}
 The ideal generated by  $I(\alpha,\Phi)$ that defines  $Y\cap X_\alpha\Fdot$ 
 is generated by at most $\|\alpha\|$ polynomials that
 are linear combinations of full rank minors of $M$.
\end{lemma}

\begin{proof}
 In the Pl\"ucker embedding of $\Gr(\ell,n)$, the Pl\"ucker coordinates
 $\{p_\beta\mid\beta\in\binom{[n]}{\ell}\}$ form a basis of the linear forms.
 If $M$ is a $\ell\times n$ matrix with $H=R_{\ell}(M)$, then $p_\beta(H)$ is
 the $\ell\times \ell$ minor of the matrix $M$ formed by its columns indexed by $\beta$.

 The Schubert variety $X_\alpha\Edot$ given by the coordinate flag $\Edot$ is defined
 in $\Gr(\ell,n)$ by the vanishing of the Pl\"ucker coordinates $p_\beta$ for
 $\beta\not\leq\alpha$. 
 If $g$ is a matrix sending $\Edot$ to $\Fdot$, then it acts on Pl\"ucker space 
 sending  $X_\alpha\Edot$ to  $X_\alpha\Fdot$ and thus  $X_\alpha\Fdot$ is defined by the
 linear forms $g(p_\beta)$ with  $\beta\not\leq\alpha$. 
 Thus $Y\cap X_\alpha\Fdot$ is defined by the linear combinations of minors
 $\{g(p_\beta)(H)\mid \beta\not\leq\alpha\}$, and this set has cardinality $\|\alpha\|$. 
\end{proof}

\begin{corollary}
 The Schubert variety $X_\alpha\Fdot$ is defined by $|\alpha|>0$ equations if and only if
 either $\ell=1$, $\ell=n{-}1$, or $\alpha=(n{-}\ell,n{-}\ell+2,\dotsc,n)$ so that $1=|\alpha|=\|\alpha\|$. 
\end{corollary}

\begin{proof}
By Lemma~\ref{L:Number_eqs}, we must describe when $|\alpha|=\|\alpha\|$.
First, observe that the poset $\binom{[n]}{\ell}$ is ranked.
This implies that $|\alpha|\leq\|\alpha\|$ 
and equality holds when the complement of the interval below
$\alpha$ in $\binom{[n]}{\ell}$ is a chain.
This also implies that every index $\beta\in\binom{[n]}{\ell}$ is comparable to $\alpha$.
When $\ell=1$ or $\ell=n{-}1$, the poset $\binom{[n]}{\ell}$ is a chain, so these
conditions hold for all $\alpha$.

Also note that if $|\alpha|<\|\alpha\|$ and $\alpha<\beta$, then we also have
$|\beta|<\|\beta\|$ as the poset $\binom{[n]}{\ell}$ is ranked.
When $1<\ell<n{-}1$, there are two indices $\alpha$ with $|\alpha|=2$, so that 
we have $|\alpha|<\|\alpha\|$ for all indices $\alpha$ with $|\alpha|>1$.
As there is a unique index $\alpha=(n{-}\ell,n{-}\ell+2,\dotsc,n)$ with $|\alpha|=1$,
for this we have  $1=|\alpha|=\|\alpha\|$.
\end{proof}

\begin{example}
 The set $\calM_{(7,8)}$ of matrices of the form $M=[X:I_2]$ for $X\in \Mat_{2\times 6}(\C)$
 parame\-trizes $X^\circ_{(7,8)}\Edot$, which is a dense subset of the
 Grassmannian $\Gr(2,8)$.
 In this cell, the Schubert variety $X_{(4,8)}\Fdot$ is defined by the vanishing of
 the $\binom{8}{6}=28$ minors of size $6\times 6$ of $\smMat{M}{\Phi_4}$ (when $i=1$), 
 and there are no equations for $i=2$.
 These minors are dependent.
 For $\Phi$ general, these 28 minors span a six-dimensional linear space of
 polynomials, by Lemma~\ref{L:Number_eqs}.
 Indeed, the set of sequences $\{\beta\in\binom{[8]}{2}\mid (4,8) \not\geq \beta\}$ 
 is 
\[
  \{(7,8)\,,\ (6,8)\,,\ 
    (5,8)\,,\ (6,7)\,,\  
    (5,7)\,,\ (5,6)\}\,. \eqno{\diamond}
\]
\end{example}

Let $(\alpha^1,\dotsc,\alpha^s)$ be a Schubert problem on $\Gr(\ell,n)$ and suppose
that $\Fdot^1,\dotsc,\Fdot^s$ are flags in general position,
giving an instance~\eqref{Eq:SchubertProblem} of this Schubert problem.
We use Propositions~\ref{P:minor_eqns} and~\ref{P:localCoordinates} to give three
formulations of this instance in terms of equations and local coordinates.
Let $\Phi^1,\dotsc,\Phi^s$ be full-rank matrices giving the flags, $\Fdot^i:=\Rdot(\Phi^i)$.

\subsubsection{Local coordinates for $\Gr(\ell,n)$.}\label{SSS:local_big}
 Let $\calM:=\calM_{(n-\ell+1,\dotsc,n)}$ be the set of matrices $(X\colon I_\ell)$ where
 $X\in\Mat_{\ell\times(n-\ell)}(\C)$, which parametrizes the dense Schubert cell
 $X_{(n{-}\ell{+}1,\dotsc,n)}^\circ\Edot$.
 Then the instance~\eqref{Eq:SchubertProblem} is defined in $\calM$ by the collection of
 ideals $I(\alpha^i,\Phi^i)$, for $i=1,\dotsc,s$.
 This uses $\ell(n{-}\ell)$ variables.

\subsubsection{Local coordinates for $X_{\alpha^s}\Fdot^s$.}\label{SSS:local_med}
 Assume that the coordinates for $\C^n$ have been chosen so that $\Fdot^s=\Edot$, the
 standard coordinate flag---this may be achieved by replacing each matrix $\Phi^i$ by
 $\Phi^i(\Phi^s)^{-1}$. 
 Let $\calM:=\calM_{\alpha^s}$ be the set of matrices parametrizing the Schubert cell
 $X^\circ_{\alpha^s}\Edot$.
 Then the instance~\eqref{Eq:SchubertProblem} is defined in $\calM$ by the collection of ideals
 $I(\alpha^i,\Phi^i)$, for $i=1,\dotsc,s{-}1$.
 This uses $\ell(n{-}\ell)-|\alpha^s|$ variables.

\subsubsection{Local coordinates for 
    $X_{\alpha^s}\Fdot^s\cap X_{\alpha^{s-1}}\Fdot^{s-1}$.}\label{SSS:local_small}
 Assume that the coordinates for $\C^n$ have been chosen so that $\Fdot^s=\Edot$ and 
 $\Fdot^{s-1}=\Epdot$ are the standard coordinate flags.
 This may be achieved by choosing $\be_i$ to be any nonzero vector
 in the one-dimensional linear subspace $F^s_i\cap F^{s-1}_{n+1-i}$ for each $i=1,\dotsc,n$.
 Then the instance~\eqref{Eq:SchubertProblem} is defined in
 $\calM=\calM_{\alpha^s}^{\alpha^{s-1}}$ by the collection of ideals $I(\alpha^i,\Phi^i)$,
 for $i=1,\dotsc,s{-}2$. 
 This uses $\ell(n{-}\ell)-|\alpha^s|-|\alpha^{s-1}|$ variables.
 (This requires that $\dim F^s_i\cap F^{s-1}_{n+1-i}=1$ for each $i$, which is equivalent to the two
   flags $\Fdot^s$ and $\Fdot^{s-1}$ being in linear general position.)

\begin{remark}
 In each of these formulations, if the flags $\Fdot^i$ are general then the set of
 solutions to the system of equations in $\calM$ will give the set of solutions to the
 Schubert problem, under the map $\calM\ni M\mapsto R_\ell(M)$. 
 Indeed, the parametrization map $\calM\to\Gr(\ell,n)$ is one-to-one on an open subset of
 $\calM$ and the image is dense in the corresponding Schubert variety or intersection of
 Schubert varieties. 
 By Kleiman's Transversality Theorem, there are no solutions in the lower-dimensional
 subvariety where the map is not injective.
\hfill{$\diamond$}
\end{remark}
\section{Primal-dual formulation of Schubert problems}\label{S:dual}

Associating a $\ell$-plane $H$ in $V=\C^n$ to its annihilator $H^\perp$ in the linear dual
\defcolor{$V^*$} of $V$ identifies the Grassmannian $\Gr(\ell,V)$ with $\Gr(n{-}\ell,V^*)$.
This identification sends Schubert varieties to Schubert varieties.
The graph of this map is defined by bilinear equations in Stiefel coordinates.
We explain how the parametrization of Schubert varieties and a twist on the classical reduction
to the diagonal leads to a formulation of any Schubert problem on $\Gr(\ell,V)$ as a square
system of bilinear equations.

\subsection{Parametrizing dual Schubert varieties}\label{SS:dual}

Let $\be_1,\dotsc,\be_n$ be a basis for the vector space $V$ and
\defcolor{$\be^*_1,\dotsc,\be^*_n$} be the dual basis for $V^*$.
Write \defcolor{$\perp$} for the canonical map $\perp\colon H\to H^\perp$ between 
$\Gr(\ell,V)$ and $\Gr(n{-}\ell,V^*)$.
We first identify the image ${\perp}(X_\alpha\Fdot)$ of a Schubert variety under the map
$\perp$. 

Each flag $\Fdot$ on $V$ has a corresponding \demph{dual flag} \defcolor{$\Fdot^\perp$} on
$V^*$, where $(\Fdot^\perp)_i=(F_{n-i})^\perp$, 
\[
   \Fdot^\perp\ \colon\ 
    (F_{n-1})^\perp\subsetneq  (F_{n-2})^\perp\subsetneq 
    \dotsb \subsetneq   (F_{1})^\perp\subsetneq V^*\,,
\]
For $\alpha\in\binom{[n]}{\ell}$, define the sequence
$\defcolor{\alpha^\perp}\in\binom{[n]}{n{-}\ell}$ by $j\in\alpha^\perp$ if and only if
$n{+}1{-}j\not\in\alpha$. 
For example, if $\ell=4$ and $n=7$, then $(2,4,5,7)^\perp=(2,5,7)$.
Note that 
\[
  [n]\ =\  \alpha^\perp\cup\{n+1-\alpha_i\mid i=1,\dotsc,\ell\}\,,
\]
and thus  $n{-}i=\#(\alpha\cap\{i{+}1,\dotsc,n\}) + \#(\alpha^\perp\cap[n{-}i])$.
(Here, $\#S$ is the cardinality of the set $S$, and 
$(\alpha\cap S):=(\alpha_i\mid\alpha_i\in S)$.)

\begin{lemma}\label{L:dualSchub}
 For a Schubert variety $X_\alpha\Fdot\subset\Gr(\ell,V)$, we have
 ${\perp}(X_\alpha\Fdot) = X_{\alpha^\perp}\Fdot^\perp$.
\end{lemma}

Note that $X_\alpha\Fdot= {\perp}(X_{\alpha^\perp}\Fdot^\perp)$.
We call $X_\alpha\Fdot$ and $X_{\alpha^\perp}\Fdot^\perp$ \demph{dual Schubert varieties}.

\begin{proof}
 Observe that if $\Fdot$ is a flag and $H$ a linear subspace, then $\dim H\cap F_{b}\geq
 a$ implies that $\dim H\cap F_{b+1}\geq a$.
 Thus the definition~\eqref{Eq:Schubert_Variety} of Schubert variety is equivalent to
 \begin{equation}\label{Eq:SchubVarAlt}
   \defcolor{X_\alpha\Fdot}\ :=\ 
   \{H\in \Gr(\ell,V)\ \mid\ 
     \dim(H\cap F_{i})\geq \#(\alpha \cap [i]),\mbox{ for  }i=1,\dotsc,n\}\,.
 \end{equation}
 For every $H\in\Gr(\ell,V)$ and all $i=1,\dotsc,n$, the following are equivalent.
 \begin{multline*}
   \dim H\cap F_i\ \geq\ \#(\alpha \cap [i])\\
  \Leftrightarrow\ 
   \makebox[330pt][l]{$\dim(\Span\{ H,F_i\})\ \leq\ \ell{+}i-\#(\alpha \cap [i])\ =\ 
    i + \#(\alpha \cap \{i{+}1,\dotsc,n\})$}\\
  \Leftrightarrow\ 
    \makebox[330pt][l]{$\dim(\Span\{ H,F_i\}^\perp)\ \geq
      n{-}i - \#(\alpha \cap \{i{+}1,\dotsc,n\})$} \\
  \Leftrightarrow\ 
    \makebox[372pt][l]{$\dim(H^\perp \cap F^\perp_{n-i})\ \geq
      n{-}i - \#(\alpha \cap \{i{+}1,\dotsc,n\})\,.$}
 \end{multline*}
 Since $n{-}i-\#(\alpha \cap \{i{+}1,\dotsc,n\})=\#(\alpha^\perp\cap[n{-}i])$, the
 lemma follows from~\eqref{Eq:Schubert_Variety}.
\end{proof}

\begin{example}
Observe that we have
\begin{equation}\label{Eq:dualCells}
    \left(\begin{array}{ccccccc} 
         \!m_{1,1}&\!1\!&  0   &\!0&0\!&0&\!0\!\\
         \!m_{2,1}&\!0\!&m_{2,3}&\!1&0\!&0&\!0\!\\
         \!m_{3,1}&\!0\!&m_{3,3}&\!0&1\!&0&\!0\!\\
         \!m_{4,1}&\!0\!&m_{4,3}&\!0&0\!&m_{4,6}&\!1\!
    \end{array}\right)
    \left(\!\begin{array}{ccc} 
         1&0&0\\
         -m_{1,1}&0&0\\
         0&1&0\\
         -m_{2,1}&-m_{2,3}&0\\
         -m_{3,1}&-m_{3,3}&0\\
          0&0&1\\
         -m_{4,1}&-m_{4,3}&-m_{4,6}
    \end{array}\!\right)
   \ =\ 
    \left(\!\begin{array}{ccc} 
     0&0&0\\ 0&0&0\\ 0&0&0\\ 0&0&0
    \end{array}\!\right)\ .
 \end{equation}
Thus if $H\in X_\alpha\Edot$ is the row space of the $4\times 7$ matrix  
$M\in\calM_{(2,4,5,7)}$ in~\eqref{Eq:dualCells}, then
$H^\perp\in X_{(2,4,5,7)^\perp}\Edot^\perp$ is the column space of the $7\times 3$ matrix
in~\eqref{Eq:dualCells} and vice-versa.
Here, the dual basis $\be^*_1,\dotsc,\be^*_n$ corresponds to the rows of the 
$7\times 3$ matrix.
In particular, $\calM^{(2,5,7)}$ parametrizes $X_{(2,4,5,7)^\perp}\Edot^\perp$ via the map
$M\mapsto C_{3}(M^T)$. 
Here, we write $M^T$ for the transpose of a matrix $M$.
\hfill{$\diamond$}
\end{example}

\begin{lemma}\label{L:dual_parametrization}
 The Schubert cell $X_{\alpha^\perp}\Edot^\perp$ is parametrized by matrices
 in $\calM^{\alpha^\perp}$ via $M\mapsto C_{n-\ell}(M^T)$.
 Similarly, the intersection 
 $X_{\alpha^\perp}\Edot^\perp\cap X_{\beta^\perp}{\Epdot}^\perp$ is parametrized by
 matrices in $\calM^{\alpha^\perp}_{\beta^\perp}$ via $M\mapsto C_{n-\ell}(M^T)$.
\end{lemma}

We leave the proof of this lemma to the reader.

Note that if $H\in \calM_\alpha$ and $K\in\calM^{\alpha^\perp}$, then each of the equations 
$HK^T=0_{\ell\times(n-\ell)}$ is either trivial or of the form $m_{i,j}+k_{t,j}=0$ as
in~\eqref{Eq:dualCells}.
Define \defcolor{$\calN_\alpha$} to be the set of $n\times (n{-}\ell)$ matrices
$\{M^T\mid M\in\calM^{\alpha^\perp}\}$.
Then $\calN_\alpha$  parametrizes $X_{\alpha^\perp}\Edot^\perp$ via $N\mapsto C_{n-\ell}(N)$.
As $\Phi\Phi^{-1}=I$, if $\Fdot=\Rdot(\Phi)$ then $(\Fdot)^\perp$ is the flag whose $i$
plane is the span of the last $i$ columns of $\Phi^{-1}$.
Thus $\calN_\alpha$ parametrizes $X_{\alpha^\perp}\Edot^\perp$ via 
$N\mapsto C_{n-\ell}(\Phi^{-1}N)$.

\subsection{The primal-dual formulation of a Schubert problem}\label{SS:DualSchubProb}

Let $\defcolor{\Delta}\colon\Gr(\ell,V)\rightarrow\Gr(\ell,V)\times\Gr(n{-}\ell,V^*)$ be the 
graph of the isomorphism $\perp\colon\Gr(\ell,V)\to\Gr(n{-}\ell,V^*)$.
We call $\Delta$ the \demph{dual diagonal} map and will identify $\Gr(\ell,V)$ with its image
under $\Delta$. 
In this context, the classical reduction to the diagonal becomes the following.

\begin{lemma}\label{L:diagonal}
 Let $A,B\subset\Gr(\ell,V)$.  
 Then $\Delta(A\cap B)= (A\times{\perp}(B))\cap\Delta(\Gr(\ell,V))$.
 When $A$ and $B$ are subschemes, this is a scheme-theoretic equality.
\end{lemma}

\begin{proof}
 Let $H\in A\cap B$.
 Then $\Delta(H)=(H,H^\perp)\in A\times{\perp}(B)$, which establishes the containment
 $\subset$.
 For the other containment, let $(H,K)\in  (A\times{\perp}(B))\cap \Delta(\Gr(\ell,V))$.
 Then $H\in A$ and $H^\perp=K\in{\perp}(B)$, so that $H\in B$, which completes the proof 
 of the equality as sets. 
 To see that this equality is scheme-theoretic, observe that pulling back the ideal of 
 $A\times{\perp}(B)$ along the map $\Delta$ gives the ideal generated by the ideals of $A$ and of $B$, which
 is the ideal of the scheme-theoretic intersection, $A\cap B$.
\end{proof}

We use Lemma~\ref{L:diagonal} to express a Schubert
problem as a complete intersection given by bilinear equations.
Suppose that $M$ is a $\ell\times n$ matrix whose row space $H$ is a $\ell$-plane in
$V$ and $N$ is a $n\times(n{-}\ell)$ matrix whose column space $K$ is a $(n{-}\ell)$-plane
in $V^*$.
(The coordinates of the matrices---columns for $M$ and rows for $N$---are with respect to 
the bases $\be_i$ and $\be^*_j$, respectively.)
Then $H^\perp=K$ if and only if $MN=0_{\ell\times(n{-}\ell)}$, giving $\ell(n{-}\ell)$ bilinear 
equations in the entries of $M$ and $N$ for $\Delta(\Gr(\ell,V))$.
We record this fact.

\begin{lemma}\label{L:bilinear}
 Let $A,B$ be two subsets of\/ $\Gr(\ell,V)$ and suppose that $\calM$ is a set of 
 $\ell\times n$ matrices parametrizing $A$ (via row span) and that $\calN$ is a set of   
 $n\times(n{-}\ell)$ matrices parametrizing ${\perp}(B)$ (via column span).
 Then the subset of $\calM\times\calN$ defined by the equations $MN=0_{\ell\times(n-\ell)}$
 parametrizes $\Delta(A\cap B)$ as a subscheme of $A\times{\perp}(B)$.
\end{lemma}

We deduce the primal-dual formulation of the Schubert variety $X_\alpha\Fdot$.

\begin{corollary}\label{C:Primal_dual_Schub_Var}
 Suppose that $\Fdot=\Rdot(\Phi)$ is a general flag. Then the bilinear equations $M\Phi^{-1}N=0_{k\times(n-\ell)}$ on pairs 
 $(M,N)\in  \calM_{(n-\ell+1,\dotsc,n)}\times\calN_\alpha$ define 
 $\Delta(X^\circ_{(n-\ell+1,\dotsc,n)}\Edot\cap X^\circ_\alpha\Fdot)$ as a system of
 $\ell(n{-}\ell)$ equations in $2\ell(n{-}\ell)-|\alpha|$ variables.
\end{corollary}

These bilinear equations define the scheme-theoretic intersection; this is just the local version of Lemma~\ref{L:diagonal}.

 We require $\Fdot$ to be sufficiently general so that
 $X^\circ_{(n-\ell+1,\dotsc,n)}\Edot\cap X^\circ_\alpha\Fdot$ is nonempty,
 for then it is a dense subset of $X_\alpha\Fdot$ (as $X^\circ_{(n-\ell+1,\dotsc,n)}\Edot$ is
 a dense subset of $\Gr(\ell,V)$).
 Since $\dim\Delta(X^\circ_{(n-\ell+1,\dotsc,n)}\Edot\cap X^\circ_\alpha\Fdot)=
        \dim(X_\alpha\Fdot)=\ell(n{-}\ell)-|\alpha|$, this system of equations and variables
 exhibits $X_\alpha\Fdot$ as a complete intersection of bilinear equations in local
 coordinates which we call the \demph{primal-dual formulation} of the Schubert variety  
 $X_\alpha\Fdot$.   

\subsection{Primal-dual formulation of a Schubert problem}
Extending this primal-dual formulation to Schubert problems uses a dual diagonal map 
to the small diagonal in a larger product of Grassmannians. 
Define
 \[
   \Delta^s\ \colon\ 
   \Gr(\ell,V)\rightarrow\Gr(\ell,V)\times \bigl(\Gr(n-\ell,V^*)\bigr)^{s-1}\,,
 \]
by sending $H\mapsto(H,H^\perp,\dotsc,H^\perp)$.

\begin{lemma}\label{L:ellDiagonal}
 Let $A_1,\dotsc,A_{s}\subset\Gr(\ell,V)$.
 Then we have the scheme-theoretic equality
\[
   \Delta^s(A_1\cap \cdots \cap A_s)\ =\ 
   (A_1\times {\perp}(A_2)\times \cdots \times {\perp}(A_s))
    \bigcap \Delta^s(\Gr(\ell,V))\,.
\]
\end{lemma}

Lemma ~\ref{L:bilinear} extends to the dual diagonal of many factors.

\begin{lemma}\label{L:ellBilinear}
 Let $A_1,\dotsc,A_{s}\subset\Gr(\ell,V)$, suppose that $\calM$ is a set of 
 $\ell\times n$ matrices parametrizing $A_1$ and $\calN_i$ is set of 
 $n\times (n{-}\ell)$ matrices parametrizing ${\perp}(A_i)$, for each $i=2,\dotsc,s$.
 Then $\Delta^s(A_1\cap \cdots \cap A_s)$ is the subscheme of 
 $A_1\times {\perp}(A_2) \times \cdots \times {\perp}(A_s)$ defined in the 
 parametrization $\calM\times\calN_2\times\cdots\times\calN_s$ by the equations
 $MN_i=0_{\ell\times(n{-}\ell)}$ for $i=2,\dotsc,s$, where 
 $(M,N_2,\dotsc,N_s)\in\calM\times\calN_2\times\cdots\times\calN_s$. 
\end{lemma}

We deduce our main theorem of this section.

\begin{theorem}\label{th:main}
 Let $(\alpha^1,\dotsc,\alpha^s)$ be a Schubert problem for $\Gr(\ell,V)$ and suppose that 
 $\Fdot^1,\dotsc,\Fdot^s$ are general flags given by matrices $\Phi^1,\dotsc,\Phi^s$.
 Then the system of equations
 \begin{equation}\label{Eq:PDF-I}
   (M\Phi^s)(\Phi^i)^{-1}N_i\ =\ 0_{\ell\times(n{-}\ell)}
   \qquad\mbox{for}\ i=1,\dotsc,s{-}1\,,
 \end{equation}
 where 
 $(M,N_1,\dotsc,N_{s-1})\in \calM_{\alpha^s}\times \calN_{\alpha^1}\times\dotsb\times \calN_{\alpha^{s-1}}$ 
 defines the instance of the Schubert problem~\eqref{Eq:SchubertProblem} as the common zeroes of 
 $(s{-}1)\cdot \ell(n{-}\ell)$ bilinear equations in  
\[
   \sum_{i=1}^s \bigl(\ell(n{-}\ell)-|\alpha^i|\bigr)
   \ =\ (s{-}1)\cdot \ell(n{-}\ell)
\]
 variables, exhibiting the Schubert problem as a square system of equations.
\end{theorem}

\begin{proof}
 An instance of the Schubert problem~\eqref{Eq:SchubertProblem} is zero-dimensional and each point occurs with multiplicity one.
 The result follows from Lemma~\ref{L:ellBilinear} in which $A_1$ is the open subset of
 $X_{\alpha^s}\Fdot^s$ parametrized by $\calM_{\alpha^s}$ and for $i<s$, $A_i$ is the open
 subset of $X_{\alpha^i}\Fdot^i$ with ${\perp}(A_i)$ parametrized by
 $\calM_{\alpha^{i{\perp}}}$.
 This gives $(s{-}1)\cdot \ell(n{-}\ell)$ bilinear equations in 
 $(s{-}1)\cdot \ell(n{-}\ell)$ variables, and the generality of the flags ensures that all
 points in the intersection of Schubert varieties~\eqref{Eq:SchubertProblem} lie in the
 intersection of Schubert cells. 
\end{proof}

\begin{remark}
 Theorem~\ref{th:main} provides a formulation of an instance of a Schubert problem 
 as a square system, to which the certification afforded by Smale's $\alpha$-theory may be
 applied.  This rectifies the fundamental obstruction to using numerical methods 
 in place of certified symbolic methods for solving Schubert problems.
\hfill{$\diamond$}
\end{remark}

\begin{remark}\label{R:structure}
 The bilinear equations~\eqref{Eq:PDF-I} are partitioned into $s{-}1$ blocks of $\ell(n{-}\ell)$
 equations.  Each block of equations is linear in the $\ell(n{-}\ell)-|\alpha^s|$ variables in $M_{\alpha^s}$,
 with the $i$th block also linear in the variables in $\calN_{\alpha^i}$, which occur in no other
 block.  This bilinear structure can be exploited in the certification via $\alpha$-theory.
 Since all equations are quadratic, the supremum defining the $\gamma$ term for $\alpha$-theory
 is taken over only one term.  
 In particular, if $G$ is the bilinear system and $x$ is a point, we have
 \begin{equation}\label{Eq:gamma}
   \gamma(G,x)\ =\ \left\|\frac{DG(x)^{-1}D^2G}{2}\right\|\ \leq\ 
     \left\|\frac{D^2G}{2}\right\|\cdot\left\|DG(x)^{-1}\right\|
 \end{equation}
 where $DG(x)$ is the Jacobian matrix of $G$ evaluated at $x$ and
 $D^2G$ is the $2^{\rm nd}$ derivative of $G$, which is constant.  
 In practice, one can either compute $\gamma$ exactly or use an easier
 formula to compute an upper bound.  The upper bound in \eqref{Eq:gamma}
 follows from the sub-multiplicative property of operator norms and allows
 one to compute $\|D^2G\|$ once.  A general upper bound for polynomial systems
 is provided in Proposition 3 from \S~I-3 of \cite{SS93}, which could possibly
 be made sharper by exploiting the bilinear structure.
 Similar bilinear structure was exploited for eigenvalue
 and generalized eigenvalue problems in \cite{A14,DS00}.
\hfill{$\diamond$}
\end{remark}

We conclude this section with a short summary of how to derive the corresponding square system from \eqref{Eq:SchubertProblem}.  
In particular, for each instance of a Schubert problem $\balpha$ involving flags
$\Fdot^1,\dotsc,\Fdot^s$, one obtains a parametrized system $G_\balpha$.  
The parameters correspond to matrices $\Phi^1,\dots,\Phi^s$ which define the flags via
$\Fdot^i=\Rdot(\Phi^i)$ for each $i=1,\dotsc,s$. 
\begin{enumerate}
\item    The rows of $\Phi^s$ give a basis of $V$, and $H\in X^\circ_{\alpha^s}\Fdot^s$ if an only
     if $H$ is the row space of a matrix $M\in \calM_{\alpha^s}$ defined relative to this basis. 
    Thus we introduce $\ell(n-\ell)-|\alpha^s|$ variables to parametrize $\calM_{\alpha^s}$.
    A change of basis to $\be_1,\dotsc,\be_n$ induces an action of $\GL(V)$ given by $M\mapsto M\Phi^s$.
    Thus $H\in X^\circ_{\alpha^s}\Fdot^s$ if and only if $H$ is the row space of $M\Phi^s$ for some
    $M\in \calM_{\alpha^s}$. 
    Note that the entries of $M\Phi^s$ are linear in the entries of $M$, and the nonconstant entries
    of $M$ are the primal coordinates for an instance of the Schubert problem $\balpha$. 
    Determining these coordinates solves the instance.

\item    For $i=1,\dotsc,s-1$, we use parameters for $\calN_{\alpha^{i\perp}}$ which are dual to
     those of $\calM_{\alpha^i}$. 
     First, we note the columns of $(\Phi^i)^{-1}$ give a basis of $V^*$, and $H\in
     X^\circ_{\alpha^i}\Fdot^i$ if and only if $H^\perp \in X^\circ_{\alpha^{i\perp}}\Fdot^{i\perp}$
     which occurs when $H^\perp$ is the column space of a matrix $N_i\in \calN_{\alpha^{i\perp}}$
     defined relative to the columns of $(\Phi^i)^{-1}$. 
     Thus we introduce $\ell(n-\ell)-|\alpha^i|$ variables to parametrize
     $\calN_{\alpha^{i\perp}}$. 
     A change of basis to $\be_1^*,\dotsc,\be_n^*$ induces an action of $\GL(V^*)$ given by
     $N_i\mapsto (\Phi^i)^{-1}N_i$, and we have $H^\perp\in X^\circ_{\alpha^{i\perp}}\Fdot^{i\perp}$
     when $H^\perp$ is the column space of $(\Phi^i)^{-1}N_i$ for some $N_i\in \calN_{\alpha^{i\perp}}$. 
     The entries of $(\Phi^i)^{-1}N_i$ are linear in $N_i$, and the nonconstant entries of $N_i$ are dual
     coordinates for an instance of a Schubert problem. 

\item  We now construct $G_\balpha$.  An $\ell$-plane $H$ in the space parametrized by the primal coordinates
       is a member of $X_{\alpha^i}\Fdot^i$ if and only if $H^\perp$ is a member of
       $X_{\alpha^{i\perp}}\Fdot^{i\perp}$.  As the bases $\be_1,\dotsc,\be_n$ and $\be_1^*,\dotsc,\be_n^*$ are
       dual, this membership is characterized by the $\ell(n{-}\ell)$ equations given by 
       $(M\Phi^s)(\Phi^i)^{-1}N_i = 0_{\ell\times(n-\ell)}$. 
       Solving the system of these equations given by all $i=1,\dotsc,s{-}1$ and projecting the solutions to the
       primal factor solves the given instance of $\balpha$.  
\end{enumerate}

\section{Improvements to the primal-dual formulation}\label{S:improve}

The formulation of a Schubert problem in Theorem~\ref{th:main} solves the
problem of certifiability, but is not particularly efficient.
For example, the Schubert problem of Example~\ref{Ex:SP} involves $3\cdot 12=36$ equations
and variables.
We present two improvements to Theorem~\ref{th:main}.
Let $\balpha=(\alpha^1,\dotsc,\alpha^s)$ be a Schubert problem for $\Gr(\ell,V)$ and
$\Fdot^1,\dotsc,\Fdot^s$ be general flags~in~$V$.

\subsection{Efficient local coordinates}\label{SS:improve}
The formulation of Theorem~\ref{th:main} may be improved using the local coordinates
$\calM_\alpha^\beta$ for $X_\alpha\Fdot\cap X_\beta\Fpdot$.
This reduces the number of equations/variables by 
$\lfloor\frac{s}{2}\rfloor \ell(n-\ell)$.

Adding a trivial condition $(n{-}\ell{+}1,\dotsc,n)$ to $\balpha$ if
necessary, we will assume that $s=2k$.
For each $i=1,\dotsc,k$, let $\Phi^i$ be a matrix such that
$\Fdot^{2i-1}=\Rdot(\Phi^i)$ and $\Fdot^{2i}=\Rdot(w_0\Phi^i)$. 
This requires that $\Fdot^{2i-1}$ and $\Fdot^{2i}$ are in linear general position.
For $\alpha,\beta\in\binom{[n]}{\ell}$ with $\alpha_i+\beta_{\ell+1-i}\leq n$ for all $i$,
let \defcolor{$\calN_\alpha^\beta$} be the set of $n\times(n{-}\ell)$ matrices 
$\{M^T\mid M\in \calM^{\alpha^\perp}_{\beta^\perp}\}$.

\begin{theorem}\label{th:main-II}
 The system of equations
 \begin{equation}\label{Eq:PDF-II}
   (M\Phi^k) (\Phi^i)^{-1}N_i\ =\ 0_{k\times(n-\ell)}
   \qquad\mbox{for}\ i=1,\dotsc,k{-}1\,,
 \end{equation}
 where 
 $(M,N_1,\dotsc,N_{k-1})\in \calM_{\alpha^{2k-1}}^{\alpha^{2k}}\times 
   \calN_{\alpha^1}^{\alpha^2}\times\dotsb\times \calN_{\alpha^{2k-3}}^{\alpha^{2k-2}}$ 
 defines the instance of the Schubert problem~\eqref{Eq:SchubertProblem} as the common zeroes of 
 $(k{-}1)\cdot \ell(n{-}\ell)$ bilinear equations in  
\[
   \sum_{i=1}^{k} \bigl(\ell(n{-}\ell)-|\alpha^{2i-1}|-|\alpha^{2i}|\bigr)
   \ =\ (k{-}1)\cdot \ell(n{-}\ell)
\]
 variables, exhibiting the Schubert problem as a square system of equations.
\end{theorem}

We omit the proof, as it is similar to that of Theorem~\ref{th:main}.

\begin{remark}
 The equations~\eqref{Eq:PDF-II} exhibit a block structure similar  to the
 equations~\eqref{Eq:PDF-I}, and the same comments made in Remark~\ref{R:structure}
 about their structure also apply here.

 We determine the reduction in the numbers of equations/variables obtained by using the 
 formulation of Theorem~\ref{th:main-II} in place of the formulation of Theorem~\ref{th:main}.
 Set $\defcolor{\nu}:=\ell(n{-}\ell)$.
 If $s=2k$ is even, then we do not add a trivial Schubert condition to $\balpha$.
 In this case the formulation of Theorem~\ref{th:main} involves $(2k{-}1)\cdot\nu$
 equations while that of Theorem~\ref{th:main-II} involves $(k{-}1)\cdot\nu$, a
 reduction of $k\nu$.

 If $s=2k{-}1$ is odd, then we add a trivial Schubert condition and the number of
 equations from Theorem~\ref{th:main} is $(2k{-}2)\cdot\nu$ and 
 $(k{-}1)\cdot\nu$ in  Theorem~\ref{th:main-II}, for a reduction of 
 $(k{-}1)\cdot\nu$.
 Thus the formulation of Theorem~\ref{th:main-II} reduces the numbers of equations/variables by
 $\lfloor\frac{s}{2}\rfloor\cdot \ell(n{-}\ell)$ from that of Theorem~\ref{th:main}.
\hfill{$\diamond$}
\end{remark}

\subsection{Exploiting hypersurface Schubert varieties}
A further improvement is possible when the Schubert problem $\balpha$ includes some
codimension 1 (hypersurface) conditions.
Let us write  \defcolor{$\I$} for the Schubert condition $(n{-}\ell,n{-}\ell+2,\dotsc,n)$, which defines
  a hypersurface Schubert variety.
Recall that the Schubert variety $X_{\sI}\Fdot$ is defined by a
single determinant.
Suppose that $\balpha$ has the form
$(\alpha^1,\dotsc,\alpha^{2k},\I,\dotsc,\I)$, where there are at least $t$
occurrences of $\I$ and $s=2k+t$.

For each $i=1,\dotsc,k$, let $\Phi^i$ be a matrix giving the pair of opposite flags 
$\Fdot^{2i-1}=\Rdot(\Phi)$ and $\Fdot^{2i}=\Rdot(w_0\Phi)$ as in Subsection~\ref{SS:improve}.
For each $j=1,\dotsc,t$, let $\Psi^j$ be a $(n{-}\ell)\times n$ matrix with
$R_{n-\ell}(\Psi^j)$
the linear subspace $F^{2k+j}_{n-\ell}$ in the flag $\Fdot^{2k+j}$.

\begin{theorem}\label{th:main-III}
 With these definitions, the system of equations
 \begin{equation}\label{Eq:PDF-III}
  \begin{array}{rcl}
   (M\Phi^k) (\Phi^i)^{-1}N_i\ =\ 0_{\ell\times(n-\ell)}\vspace{5pt}
   &\mbox{for}& i=1,\dotsc,k{-}1\,,\\
   \det\left(\begin{array}{c}M\\ \Psi^j\end{array}\right)\ =\ 0
   &\mbox{for}& j=1,\dotsc,t\,,
  \end{array}
 \end{equation}
 where 
 $(M,N_1,\dotsc,N_{k-1})\in \calM_{\alpha^{2k-1}}^{\alpha^{2k}}\times 
   \calN_{\alpha^1}^{\alpha^2}\times\dotsb\times\calN_{\alpha^{2k-3}}^{\alpha^{2k-2}}$ 
 defines the instance of the Schubert problem~\eqref{Eq:SchubertProblem} as the common zeroes of 
 $(k{-}1)\cdot \ell(n{-}\ell)$ bilinear equations and $t$ determinantal equations in  
\[
   \sum_{i=1}^{k} \bigl(\ell(n{-}\ell)-|\alpha^{2i-1}|-|\alpha^{2i}|\bigr)
   \ =\ k\cdot \ell(n{-}\ell) - \sum_{i=1}^{2k} |\alpha^i|
   \ =\ (k{-}1)\cdot \ell(n{-}\ell) + t
\]
 variables, exhibiting the Schubert problem as a square system of equations.
\end{theorem}

\begin{remark}
The previous comments based on the block structure of equations remain valid for the
formulation~\eqref{Eq:PDF-III}. 

Suppose now that $\balpha$ has $a$ conditions $\alpha^1,\dotsc,\alpha^a$ with $|\alpha^i|>1$
for $i=1,\dotsc,a$ and $b\geq 1$ occurrences of the codimension 1 condition $\I$.
If $a=2k$ is even, then we apply Theorem~\ref{th:main-III} with $t=b$.
If $a=2k{-}1$ is odd, then we set $\alpha^{2k}=\I$ and apply Theorem~\ref{th:main-III} with
$t=b{-}1$ (necessarily, $b\geq 2$ for this). 
Table~\ref{T:savings} shows the gain in efficiency from using Theorem~\ref{th:main-III} 
in place of Theorem~\ref{th:main-II} when we have hypersurface Schubert conditions,
recording the reduction in the number of equations/variables.
Set $\nu:=\ell(n{-}\ell)$.
In all cases we have a net reduction.
\hfill{$\diamond$}
\end{remark}

\begin{table}[htb]
 \caption{Efficiency gain with hypersurfce Schubert conditions}\label{T:savings}

 \begin{tabular}{|c||c|c|c|c|}\hline
  $(a,b)\mod 2$&$(0,0)$&$(0,1)$&$(1,0)$&$(1,1)$\raisebox{-5pt}{\rule{0pt}{17pt}}\\\hline
  reduction  &$\lfloor\frac{b}{2}\rfloor \nu-b$
             &$\lfloor\frac{b+1}{2}\rfloor \nu-b$
             &$\lfloor\frac{b}{2}\rfloor \nu-b+1$\raisebox{-6pt}{\rule{0pt}{18pt}}
             &$\lfloor\frac{b}{2}\rfloor \nu-b+1$\\\hline
 \end{tabular}
\end{table}

\begin{example}\label{Ex:437}
 We compare all these formulations for  a Schubert problem in $\Gr(3,9)$.
 Let $\beta:=(4,8,9)$, which has $|\beta|=3$ and corresponds to the condition that
 a 3-plane $H$ meets a fixed 4-plane nontrivially.
 Set $\balpha := (\beta,\beta,\beta,\beta,\I,\I,\I,\I,\I,\I)$, which has 437
 solutions. 

 The Grassmannian $\Gr(3,9)$ is defined in Pl\"ucker space $\P^{83}$ by 1050 independent
 quadratic Pl\"ucker equations.
 Sinice $\|\beta\|=10$ and $\|\I\|=1$, any instance of the Schubert problem $\balpha$ is
 defined in Pl\"ucker space by 1050 quadratic and 46 linear equations.

 The formulation of \S~\ref{SSS:local_big} in local coordinates 
 $(X\colon I_a)$ for $\Gr(3,9)$ with $X\in\Mat_{3\times 6}$ uses 18 variables and 46
 independent cubic minors/determinants to express the Schubert problem $\balpha$.
 Using the local coordinates $\calM_\beta$ as in \S~\ref{SSS:local_med}, reduces this to
 15 variables and 36 cubic equations, while the formulation of \S~\ref{SSS:local_small}
 using local coordinates $\calM_\beta^\beta$ involves 12 variables and 26 cubic equations.
 These formulations are all overdetermined.

 In contrast, the formulation of Theorem~\ref{th:main} for $\balpha$ uses $9\cdot 18=162$
 variables and 162 bilinear equations, that of Theorem~\ref{th:main-II} uses 
 $4\cdot 18= 72$ variables and bilinear equations, while that of Theorem~\ref{th:main-III} 
 uses $24$ variables with 18 bilinear equations and six cubic determinants, coming from the six
 codimension one conditions $\I$.
 
 To test the square formulation of $24$ variables, we used Bertini \cite{BHSW06} to solve an instance of this Schubert problem given by random real flags.
 The computation consumed approximately $20.37$ gigaHertz-hours of processing power to calculate $437$ approximate solutions, and the output suggests that $85$ solutions are real $3$-planes while the other $352$ are nonreal.
 For this, we used regeneration which is not specific to Schubert calculus and which may be done efficiently in parallel.

 We used rational arithmetic in alphaCertified \cite{alphaCertified} to prove that the $437$ solutions given by Bertini are approximate solutions to the computed instance, and that they correspond to distinct solutions.
 Since they are distinct, theorems of Schubert calculus guarantee that we have approximations for every solution to the instance.
 We also verified that $85$ of the approximate solutions correspond to real solutions.
 This calculation took approximately $2.00$ gigaHertz-hours of processing power.
 Certification may be done in parallel, efficiently using up to $437$ processors.
\hfill{$\diamond$}
\end{example}

\section{Primal-dual formulation for flag manifolds}\label{S:FlagVariety}

A flag manifold is the set of all flags of linear subspaces of specified dimensions, and
is therefore a generalization of the Grassmannian.
Flag manifolds have Schubert varieties and Schubert problems, and these may be formulated
as systems of determinantal equations in local coordinates.
Like the Grassmannian, there is a primal-dual formulation of Schubert problems on flag
manifolds.

Let $V$ be an $n$-dimensional complex vector space and let 
$\defcolor{\adot} \colon 0<a_1<\dotsb<a_t<n$ be a sequence of integers.
A (partial) \demph{flag of type $\adot$} is an increasing sequence
\[
   \defcolor{V_{\adot}}\ \colon\ 
    V_{a_1}\ \subsetneq\ V_{a_2}\ \subsetneq\ \dotsb\ \subsetneq\ 
    V_{a_t}\ \subsetneq\ V\,,
\]
of linear subspaces of $V$ with $\dim V_{a_i}=a_i$ for $i=1,\dotsc,t$.
A flag is \demph{complete} if $\adot=[n{-}1]$.
The set of all flags of type $\adot$ forms the flag manifold \defcolor{$\Fl(\adot,V)$},
which is a homogeneous space for the general linear group $\GL(V)$.
If we set $a_0:=0$, then it has dimension 
\[
    \defcolor{\dim\adot}\ :=\ 
      \sum_{i=1}^t (a_i-a_{i-1})(n-a_i)\,.
\]
If $\defcolor{P_{\adot}}$ is the subgroup stabilizing a flag of type $\adot$,
then $\Fl(\adot,V)\simeq GL(n,\C)/P_{\adot}$.
The Weyl group $W_{\adot}$ of $P_{\adot}$ is the Young subgroup 
$S_{a_1}\times S_{a_2-a_1}\times\dotsb\times S_{n-a_t}$ of the symmetric group $S_n$,
which is the Weyl group of $GL(n,\C)$.

Given a complete flag $\Fdot$, the flag manifold has a decomposition into Schubert cells whose
closures are Schubert varieties.
A Schubert cell consists of all flags $V_{\adot}$ having the same specified position with
respect to the reference flag $\Fdot$.
These positions are indexed by permutations $w\in S_n$ on $n$ letters with descents in the set
$\adot$, 
\[
   \defcolor{W^{\adot}}\ :=\ \{ w\in S_n\;\mid\;
     w(i)<w(i{+}1)\quad\mbox{if}\quad i\not\in\adot\}\,,
\]
which are also minimal length representatitives of cosets of $W_{\adot}$ in $S_n$.
The Schubert variety associated to a (complete) flag $\Fdot$ and permutation 
$w\in W^{\adot}$ is
 \begin{equation}\label{Eq:SchubertVarietyTypeA}
   \defcolor{X_w\Fdot}\ :=\ 
   \{ V_{\adot}\in\Fl(\adot,V)\;\mid\;
      \dim V_{a_i}\cap F_j\ \geq\ \#\{k\leq a_i\mid w(k)\leq j\}\}\,.
 \end{equation}
 For $a\in\adot$ we have the map $w\mapsto w|_a:=\{w(1),\dotsc,w(a)\}$ which sends $W^{\adot}$
 to $\binom{[n]}{a}$.
 An equivalent form of the definition~\eqref{Eq:SchubertVarietyTypeA} is that 
 \begin{equation}\label{Eq:GrassProjection}
   V_{\adot}\in X_w\Fdot\quad\mbox{if and only if}\quad
   V_a\in X_{w|_a}\Fdot\subset\Gr(a,V)\quad\mbox{for all } a\in\adot\,.
 \end{equation}
 When $t=1$ so that $\adot=\{a\}$, then $\Fl(\adot,V)=\Gr(a,V)$, permutations in $W^{\{a\}}$ are
  in bijection with $\binom{[n]}{a}$ via $w\mapsto w|_a$, and we obtain the same
  Schubert varieties as before.

As explained in~\cite[Ch.~10]{Fulton}, matrices in partial echelon form parametrize
Schubert cells. 
For $w\in W^{\adot}$ let \defcolor{$\calM_w$} be the set of those $a_t\times n$ matrices
$(m_{i,j})$ where for all $i,j$, 
 \begin{eqnarray*}
   m_{i,j}&=&1\qquad\mbox{\rm if } j=w(i)\\
   m_{i,j}&=&0\qquad\mbox{\rm if } j>w(i)\quad\mbox{\rm or}\quad
             j=w(k)\ \mbox{\rm for some }k<i\,.
 \end{eqnarray*}
Note then that $m_{i,j}$ is $1$ or $0$ unless $j=w(k)<w(i)$ for some $k>i$.
Thus $\calM_w\simeq\C^{\ell(w)}$, where $\ell(w)$ is the length of $w$, which is 
\[
   \defcolor{\ell(w)}\ :=\ \#\{k>i\mid w(k)<w(i)\}\,.
\]
For example, here are $\calM_{463512}$ with $\adot=2<4$ and $\calM_{3625714}$ for
$\adot=2<5$ 
\[
  \left(\begin{array}{cccccc}
   m_{1,1}&m_{1,2}&m_{1,3}& 1 &   0   & 0 \\
   m_{2,1}&m_{2,2}&m_{2,3}& 0 &m_{2,5}& 1 \\
   m_{3,1}&m_{3,2}&  1  & 0 &  0  & 0 \\
   m_{4,1}&m_{4,2}&  0  & 0 &  1  & 0 
     \end{array}\right)
   \qquad
  \left(\begin{array}{ccccccc}
    m_{1,1}&m_{1,2}& 1 &   0   &  0  &  0 & 0\\
    m_{2,1}&m_{2,2}& 0 &m_{2,4}&m_{2,5}& 1  & 0\\
    m_{3,1}&  1  & 0 &  0  &  0  & 0 & 0 \\
    m_{4,1}&  0  & 0 &m_{4,4}&  1  & 0 & 0 \\
    m_{5,1}&  0  & 0 &m_{5,4}&  0  & 0 & 1 
  \end{array}\right)
\]

Recall that given a full rank $a_t\times n$ matrix $M$, $R_{\adot}(M)$ is the flag 
of type $\adot$ whose $a_i$-dimensional linear subspace is the span of the first $a_i$
rows of $M$.
Under the map $M\mapsto R_{\adot}(M)$, the set $\calM_w$ of partial echelon matrices
parametrizes the Schubert cell $X_w^\circ\Edot$.
This extends to any flag, not just $\Edot$.

\begin{proposition}
 Let $\Phi$ be a full rank $n\times n$ matrix.
 For $w\in W^{\adot}$, the set $\calM_w$ of partial echelon matrices parametrizes the Schubert cell
 $X_w^\circ\Rdot(\Phi)$ via the map $M\mapsto R_{\adot}(M\Phi)$.
\end{proposition}

Let $\Phi$ be a full rank $n\times n$ matrix and let $M$ be a full rank $a_t\times n$ matrix.
By the definition~\eqref{Eq:SchubertVarietyTypeA}, for any $w\in W^{\adot}$ the partial flag
$R_{\adot}(M)$ lies in $X_w\Rdot(\Phi)$ if and only if, for all $i=1,\dotsc,r$ and
$j=1,\dotsc,n$, we have that 
 \begin{equation}\label{Eq:RankConditions}
   \rank \left(\begin{array}{c}M_{a_i}\\\Phi_j\end{array}\right)\ \leq\ 
    a_i+j-\#\{k\leq a_i\mid w(k)\leq j\}\,.
 \end{equation}
Set $\defcolor{r_w(a_i,j)}:=a_i+j-\#\{k\leq a_i\mid w(k)\leq j\}$, which is this bound on rank.
We use~\eqref{Eq:RankConditions} to give the traditional determinantal equations defining
a Schubert variety.

\begin{proposition}\label{P:minor_eqns_flag}
 Let $\calM$ be a set of matrices parametrizing a subset $Y$ of the flag manifold
 $\Fl(\adot,n)$, $\Phi$ a full-rank matrix, and $w\in W^{\adot}$.
 The intersection of\/ $Y$ with the Schubert variety $X_{w}\Rdot(\Phi)$ is defined 
 in the coordinates $\calM$ for $Y$ by the vanishing of the square $r_w(a_i,j){+}1$ 
 minors of\/ $\smMat{M_{a_i}}{\Phi_{j}}$, for $j=1,\dotsc,t$ and $k=1,\dotsc,n$.
\end{proposition}

The codimension of $X_w\Fdot$ is $\defcolor{|w|}:=\dim(\adot)-\ell(w)$.
A \demph{Schubert problem} on $\Fl(\adot,V)$ is a list $w^1,\dotsc,w^s$ of permutations in
$W^{\adot}$ such that $|w^1|+\dotsb+|w^s|=\dim(\adot)$.
By Kleiman's Transversality Theorem, if $\Fdot^1,\dotsc,\Fdot^s$ are general flags then
the intersection
 \begin{equation}\label{Eq:Flag_SP}
  X_{w^1}\Fdot^1\,\cap\,
  X_{w^2}\Fdot^2\,\cap\, \dotsb \,\cap\,
  X_{w^s}\Fdot^s
 \end{equation}
either is empty or is transverse and consists of finitely many points.
As in Section~\ref{S:Schubert_Calculus}, we may formulate a Schubert problem as a system
of determinantal equations (given by Proposition~\ref{P:minor_eqns_flag}) for membership
in $X_{w^i}\Fdot^i$ for $i=1,\dotsc,s{-}1$ in the system $\calM_{w^s}$ of local coordinates
for $X^\circ_{w^s}\Fdot^s$.
This will be overdetermined unless $|w^i|=1$ for $i=1,\dotsc,s{-}1$.

\subsection{Primal-dual formulation for Schubert problems in $\Fl(\adot,V)$}

Given a sequence $\adot\subset[n{-}1]$ , let \defcolor{$\adot^r$} be its reversal
\[
   \adot^r\ \colon\  n{-}a_t\ <\ n{-}a_{t-1}\ <\ \dotsb\ <\ 
                     n{-}a_2\ <\ n{-}a_1\,.
\]
When $n=9$, we have $(2,3,5)^r=(4,6,7)$.
The annihilator map of Section~\ref{S:dual} gives an isomorphism
$\perp\colon\Fl(\adot,V)\xrightarrow{\,\sim\,}\Fl(\adot^r,V^*)$ in which
${\perp}(V_{\adot})_{a^r_j}=(V_{a_{t+1-j}})^\perp$.
Let $w_0\in S_n$ be the permutation so that $w_0(i)=n{+}1{-}i$.
Note that $\ell(w)=\ell(w_0ww_0)$ and $w\in W^{\adot}$ if and only if 
$w_0ww_0\in W^{\adot^r}$.
We give the version of Lemma~\ref{L:dualSchub} for $\Fl(\adot,V)$.

\begin{lemma}
 For a Schubert variety $X_w\Fdot\subset\Fl(\adot,V)$, we have 
 ${\perp}(X_w\Fdot)=X_{w_0ww_0}\Fdot^\perp$.
\end{lemma}

For $w\in W^{\adot}$, define \defcolor{$\calN_w$} to be the set of 
$n\times(n{-}a_1)$-matrices,
\[
   \calN_w\ :=\ \{ (M\cdot w_0)^T\;\mid\; M\in\calM_{w_0ww_0}\}\,,
\]
which are obtained by first reversing the columns of a matrix in $\calM_{w_0ww_0}$ and then
transposing the result.
We state the analog of Lemma~\ref{L:dual_parametrization} in this context.

\begin{lemma}
 Let $\Phi$ be a nonsingular $n\times n$ matrix.
 Under the map $N\mapsto C_{\adot^r}(\Phi^{-1}N)$ the set $\calN_w$ of matrices for
 $w\in W^{\adot}$ parametrizes the Schubert cell $X_{w_0ww_0}^\circ(\Rdot(\Phi))^\perp$ in 
 $\Fl(\adot^r,V^*)$.  
\end{lemma}

Let us consider this duality in terms of bilinear equations in Stiefel coordinates.
Suppose that $M$ is a full rank $a_t\times n$ matrix with corresponding flag
$V_{\adot}:=R_{\adot}(M)$ and $N$ a full rank $n\times(n{-}a_1)$ matrix with corresponding flag
$\defcolor{U_{\adot^r}}:=C_{\adot^r}(N)$.
As in Section~\ref{S:dual}, the columns of $M$ correspond to a basis $\be_1,\dotsc,\be_n$ of
$V$ and the rows of $N$ to the dual basis $\be_1^*,\dotsc,\be_n^*$ of $V^*$.
Then the condition that $U_{\adot^r}=V_{\adot}^\perp$ is that
\[
   U_{a^r_{t+1-j}}\ =\ U_{n-a_j}\ =\ (V_{a_j})^\perp
   \qquad\mbox{for } j=1,\dotsc,t\,.
\]
In terms of the matrices $M$ and $N$ this is 
 \begin{equation}\label{Eq:ZeroProduct}
   M_{a_j}\cdot N'_{n{-}a_j}\ =\ 0_{a_j\times(n{-}a_j)}
   \qquad\mbox{for } j=1,\dotsc,t\,.
 \end{equation}
(Here $N'_{n{-}a_j}$ is the matrix formed by the first $n{-}a_j$ columns of $N$.) 
This set of $\sum_j a_j(n{-}a_j)$ bilinear equations is redundant.
Write \defcolor{$\mbox{\rm row}_i(M)$} for the $i$th row of the matrix $M$ and 
\defcolor{$\mbox{\rm col}_j(N)$} for the $j$th column of $N$.
The following lemma is easily verified.

\begin{lemma}
 The set of bilinear equations~\eqref{Eq:ZeroProduct} for 
 $(M,N)\in\Mat_{a_t\times n}(\C)\times \Mat_{n\times(n-a_1)}(\C)$ consists of the 
 equations 
\[
    \mbox{\rm row}_i(M)\cdot\mbox{\rm col}_j(N)\ =\ 0\,,
\]
 for those $(i,j)$ such that there is some $k$ with $i\leq a_k$ and 
 $j\leq(n{-}a_k)$.
 There are exactly $\dim\adot$ such equations.
\end{lemma}

Write \defcolor{$MN=0_{\adot}$} for this set of $\dim \adot$ bilinear equations.
We give the primal-dual formulation of Schubert problems in $\Fl(\adot,V)$, which follows
from the reduction to the dual diagonal as in Section~\ref{S:dual}.

\begin{theorem}\label{th:main_A}
 Let $(w^1,\dotsc,w^s)$ be a Schubert problem for $\Fl(\adot,V)$ and suppose that
 $\Fdot^1,\dotsc,\Fdot^s$ are general flags given by matrices $\Phi^1,\dotsc,\Phi^s$ so
 that $\Fdot^i=\Rdot(\Phi^i)$. 
 Then the system of equations
 \begin{equation}\label{Eq:Flag_PDF}
   (M\Phi^s)(\Phi^i)^{-1}N_i\ =\ 0_{\adot}
   \qquad\mbox{for}\ i=1,\dotsc,s{-}1\,,
 \end{equation}
 where 
 $(M,N_1,\dotsc,N_{s-1})\in\calM_{w^s}\times\calN_{w^1}\times\dotsb\times\calN_{w^{s-1}}$ 
 defines the instance of the Schubert problem~\eqref{Eq:Flag_SP} as the common zeroes of 
 $(s{-}1)\cdot \dim(\adot)$ bilinear equations in  $(s{-}1)\cdot \dim(\adot)$
 variables, exhibiting the Schubert problem as a square system of equations.
\end{theorem}

As in Corollary~\ref{C:Primal_dual_Schub_Var}, these are scheme-theoretic equations for the Schubert problem in local coordinates.

As in Section~\ref{S:improve}, there are some improvements  to this
formulation. 
Rather than formulate a general result along the lines of Theorem~\ref{th:main-III}, we will
make a series of remarks indicating some reductions that are possible. 

For $\alpha\in\binom{[n]}{a}$ with $a\in\adot$, there are many permutations
$w\in W^{\adot}$ with $w|_a=\alpha$.
(These are in  bijection with 
 $W_{\{a\}}/W_{\adot}=(S_a\times S_{n-a})/(S_{a_1}\times S_{a_2-a_1}\times\dotsb\times
 S_{n-a_t})$.)
Write \defcolor{$w(\alpha,\adot)$}, or \defcolor{$w(\alpha)$} when $\adot$ is assumed, for
the longest such permutation in $W^{\adot}$.
For example, if $n=8$ and $\adot=(2<4<5)$, then $w(2,5,6,8)=68{.}25{.}7{.}134$.
Observe that $|\alpha|=|w(\alpha)|$.

Schubert varieties given by permutations $w(\alpha)$ 
have a form which we exploit to reduce the number of equations and variables
in~\eqref{Eq:Flag_PDF}. 
By~\eqref{Eq:GrassProjection}, if $\alpha\in\binom{[n]}{a}$ where $a\in\adot$,
then 
\[
   X_{w(\alpha,\adot)}\Fdot\ =\ \{V_{\adot}\in\Fl(\adot,V)\;\mid\;  V_a\in X_{\alpha}\Fdot\}\,.
\]
In particular, we have
\[
  V_{\adot}\;\in\; X_{w(\alpha,\adot)}\Fdot\ \Longleftrightarrow\ 
    (V_a)^\perp\;\in\; X_{\alpha^\perp}\Fdot^\perp\,.
\]

\begin{remark}\label{R:RedOne}
  Suppose that one permutation in Theorem~\ref{th:main_A}, say $w^i$, has the form
  $w(\alpha,\adot)$ for $\alpha\in\binom{[n]}{a}$ with $a\in\adot$.
  Then we may replace the factor $\calN_{w^i}$ by the space of $n\times(n{-}a)$-matrices
  $\calN_\alpha$ and the $i$th set of equations in~\eqref{Eq:Flag_PDF} by 
\[
   (M_a\Phi^s)(\Phi^i)^{-1}N\ =\ 0_{a\times(n-a)}\,,
\]
  where $N\in\calN_{\alpha}$.
  This uses $\dim\adot-a(n{-}a)$ fewer equations and variables.
\hfill{$\diamond$}
\end{remark}

\begin{remark}\label{R:Grass_coords}
 Suppose that two of the permutations have this form, say $w^i=w(\alpha,\adot)$ and
 $w^j=w(\beta,\adot)$ for $\alpha,\beta\in\binom{[n]}{a}$ with $a\in\adot$.
 There is a matrix $\Phi$ so that $\Fdot^i=\Rdot(\Phi)$ and $\Fdot^j=\Rdot(w_0\Phi)$ are the
 pair of opposite flags associated to $\Phi$.
 In Theorem~\ref{th:main_A} we can replace the factors
 $\calN_{w^i}\times\calN_{w^j}$ by the factor $\calN_{\alpha}^{\beta}$ and the equations
 indexed by $i$ and $j$ in~\eqref{Eq:Flag_PDF} by the equations
\[
   (M_a\Phi^s)\cdot(\Phi^{-1}N)\ =\ 0_{a\times(n-a)}\,,
\]
 where $N\in\calN_{\alpha}^{\beta}$.
 This uses $2\dim\adot-a(n{-}a)$ fewer equations and variables.
\hfill{$\diamond$}
\end{remark}

\begin{remark}\label{R:codimOne}
 If some permutation $w^i$ has $|w^i|=1$, then it has the form
 $w(\alpha,\adot)$ where $\alpha\in\binom{[n]}{a}$ satisfies $|\alpha|=1$ so that
 $\alpha=\I$. 
 Write $w(\I,a,\adot)$ or $w(\I,a)$ for this permutation.
 We may remove the factor $\calN_{w^i}$ 
 and replace the $i$th set of equations in~\eqref{Eq:Flag_PDF}
 with the single determinantal equation
\[
   \det\left(\begin{array}{c} M_aF^s\\\Phi^i_{n-a}\end{array}\right)\ =\ 0\,.
\]
 This uses $\dim\adot-1$ fewer equations and variables.
\hfill{$\diamond$}
\end{remark}

\begin{remark}\label{R:bdot}
 We generalize Remark~\ref{R:RedOne}.
 Let $\bdot\subsetneq\adot$ and suppose that $v\in W^{\bdot}$.
 Write $\defcolor{w(v,\adot)}\in W^{\adot}$ for the longest permutation in $W^{\adot}$ 
 that lies in the coset $v W_{\bdot}$.
 Then $|w(v,\adot)|=|v|$, or more precisely
 $\dim\adot-\ell(w(v,\adot))=\dim\bdot-\ell(v)$.
 For example, if $n=8$, $\adot=(2<4<5)$ and 
 $v=68257134\in W^{(2<5)}$, then $w(v,\adot)=68.57.2.134$.
 Also, if $v=47681235\in W^{(2<4)}$, then 
 $w(v,\adot)=47.68.5.123$, and if $v=34685127\in W^{(4<5)}$, then
 $w(v,\adot)=68.34.5.127$.

 A flag $V_{\adot}\in X_{w(v,\adot)}\Fdot$ if and only if the subflag $V_{\bdot}$
 consisting of the subspaces with dimensions in $\bdot$ lies in 
 the Schubert variety $X_v\Fdot$ of $\Fl(\bdot,V)$.
 In this case we have that 
\[
  V_{\adot}\;\in\; X_{w(v,\adot)}\Fdot\ \Longleftrightarrow\ 
    (V_{\bdot})^\perp\;\in\; X_{w_0vw_0}\Fdot^\perp\;\subset\ \Fl(\bdot^r,V^*)\,.
\]

 If one of the permutations $w^i$ in Theorem~\ref{th:main_A} has
 the form $w(v,\adot)$ for $v\in W^{\bdot}$, then we may replace the set
 $\calN_{w^i}$ of matrices parametrizing 
 $X_{w_0w^iw_0}\Fdot^\perp$ by the set
 of matrices $\calN_v$ parametrizing $X_{w_0vw_0}\Fdot^\perp\subset\Fl(\bdot^r,V^*)$, and
 the $i$th set of equations in~\eqref{Eq:Flag_PDF} by
\[
   (M\Phi^s)(\Phi^i)^{-1}N\ =\ 0_{\bdot}\,,
\]
 for $N\in N_v$.
 This uses $\dim \adot-\dim \bdot$ fewer equations and variables.
\hfill{$\diamond$}
\end{remark}

\begin{example}
 Suppose $n=8$ and $\adot=(2<4<5)$.
 Note that 
\[
    \dim\adot\ =\ 2(8{-}2)+(4{-}2)(8{-}4)+(5{-}4)(8{-}5)\ =\ 23\,.
\]
 Using $(w)^m$ to indicate that $w$ is repeated $m$ times, consider the Schubert problem,
\[
   48573126\,,\ 
   (78453126)^2\,,\ 
   (68574123)^2\,,\ 
   (78465123)^3\,,\ 
   47385126\,,
\]
 with 128 solutions.
 The first permutation is $w(v)$ where $v=48.357.126\in W^{(2<5)}$, the
 next two are $w(\alpha)$ where $\alpha=(3,4,6,7,8)$, the next two are 
 $w(\I,2)$, and the next three are $w(\I,4)$.
 Only the last permutation does not reduce to a smaller flag manifold.

 By Theorem~\ref{th:main_A}, this has a formulation as a square system of 
 $8\cdot 23=184$ bilinear equations and variables.

 However, we may use the reduction of Remark~\ref{R:bdot} with the first permutation to
 use $2=\dim\adot-\dim\bdot$ fewer equations, then the reduction of
 Remark~\ref{R:Grass_coords} with the second and third permutations to use
 $2\cdot\dim\adot-5(8-5)=31$ fewer equations, and then the reduction of
 Remark~\ref{R:codimOne} on the next five permutations to use $5(23{-}1)=110$ fewer
 equations, and thereby obtain a square system with $184-2-31-110=41$ variables, $36$ bilinear
 equations, and $5$ determinantal equations involving $2$ quadratic determinants and $3$
 quartic determinants.

 We used Bertini \cite{BHSW06} to solve a random real instance of this Schubert problem with
 the square formulation of $41$ variables. 
 This consumed about $2.95$ gigaHertz-days of processing power to calculate $128$
 approximate solutions, and the output suggests that $42$ solutions are real flags while
 the other $86$ are nonreal. 
 As in Example \ref{Ex:437}, this computation may be done in parallel efficiently.

 We used rational arithmetic in alphaCertified \cite{alphaCertified} to prove that the $128$
 solutions given by Bertini are approximate solutions to the computed instance, and that they
 correspond to distinct solutions. 
 As in Example \ref{Ex:437}, available theorems guarantee that we found approximations for every solution to the instance.
 We also verified that $42$ of the approximate solutions correspond to real solutions.
 This calculation took $1.78$ gigaHertz-hours of processing time.
 This computation may also be done efficiently in parallel.
 Details for these computations and those in Example \ref{Ex:437} may be found
  at \url{http://www.unk.edu/academics/math/_files/primal-dual.html}.
\hfill{$\diamond$}
\end{example}
\providecommand{\MR}{\relax\ifhmode\unskip\space\fi MR }
\providecommand{\MRhref}[2]{%
  \href{http://www.ams.org/mathscinet-getitem?mr=#1}{#2}
}
\providecommand{\href}[2]{#2}

\end{document}